\documentclass{article}
\usepackage{amsmath,amssymb,amsfonts,amsthm,latexsym,amstext,amscd}
\usepackage{epsfig,graphics,graphicx,color}
\usepackage{enumerate}
\usepackage{hyperref}

\newcommand{\bsx}{\boldsymbol{x}}%
\newcommand{\bsy}{\boldsymbol{y}}%
\newcommand{\bsz}{\boldsymbol{z}}%
\newcommand{\measure}[1]{\lambda_{s}(#1)}
\newcommand{\NN}{\mathbb N}
\newcommand{\ZZ}{\mathbb Z}
\newcommand{\RR}{\mathbb R}
\newcommand{\QQ}{\mathbb Q}

\renewcommand{\(}{\left (}
  \renewcommand{\)}{\right )}
\newcommand{\digit}[2]{a_{#1}(#2)}
\newtheorem{theorem}{Theorem}
\newtheorem{lemma}{Lemma}

\newtheorem{prop}{Proposition}

\begin{document}

\title{A metric result for special subsequences of the Halton sequences}
\author       {Roswitha Hofer\thanks{supported by the Austrian Science Fund (FWF):
Project F5505-N26, which is a part of the Special Research Program
``Quasi-Monte Carlo Methods: Theory and Applications''} }%insert author(s)

\maketitle

\begin{abstract}
In this paper we investigate the special subsequence of the Halton sequence indexed by $\lfloor n\beta\rfloor$ with $\beta\in\RR$ and prove a metric almost low-discrepancy result. 
\end{abstract}

%\begin{centerline}%
%{\dedicatory\textsl{Dedicated to the memory of }}
%\end{centerline}


\section{Introduction}

The \emph{discrepancy} of the first $N$ terms of any sequence 
$(\bsz_n)_{n \ge 0}$ of points in $[0,1)^s$ is defined by
\begin{equation*}
  D_N(\bsz_n)=\sup_{J}\left |\frac{A_N(J)}{N}-\measure{J} \right |,
\end{equation*}
where the supremum is extended over all half-open subintervals $J$ of $[0,1)^s$, 
$\lambda_s$ denotes the $s$-dimensional Lebesgue measure,
and the counting function $A_N(J)$ is given by
\begin{equation*}
  A_N(J)=\#\{0\le n\le N-1\ :\ \bsz_n\in J\}.
\end{equation*}
For the sake of simplicity, we will sometimes write $D_N$ instead of $D_N(\bsz_n)$. If the supremum is extended over all half-open subintervals $J$ of $[0,1)^s$ with the lower left point in the origin then we arrive at the notion of the \emph{star-discrepancy} $D_N^*$. It is not so hard to see that $D_N^*\leq D_N\leq 2^s D_N^*$. A sequence $(\bsz_n)_{n \ge 0}$ of points in $[0,1)^s$ is called uniformly distributed if $\lim_{N\to\infty}D_N=0$. 

It is frequently conjectured in the theory of irregularities of distribution, that for every sequence $(\bsz_n)_{n\geq 0}$ in $[0,1)^s$ we have 
$$D_N\geq c_s\frac{\log^sN}{N}$$
for a constant $c_s>0$ and for infinitely many $N$. Therefore sequences whose discrepancy satisfies $D_N\leq C \log^sN/N$ for all $N$ with a constant $C>0$ that is independent of $N$ (or $D_N=O( \log^sN/N$)), are called \emph{low-discrepancy sequences}. Well-known examples of low-discrepancy sequences are, for example, the $s$-dimensional Halton sequences and one-dimensional Kronecker sequences $(\{n\alpha\})_{n\geq 0}$ with $\alpha$ irrational and having bounded continued fraction coefficients. 

For an integer $b\ge 2$, let $\ZZ_b=\{0,1,\ldots,b-1\}$ denote the least residue
system modulo $b$. Let $n=\sum_{j=1}^{\infty}\digit{j}{n}b^{j-1}$ with all
$\digit{j}{n}\in\ZZ_b$ and $\digit{j}{n}=0$ for all sufficiently large
$j$ be the unique digit expansion of the integer $n\ge 0$ in base $b$. 
The \textit{radical-inverse function $\phi_b$} in base $b$ is defined by 
\begin{equation*}
  \phi_b(n)=\sum_{j=1}^{\infty}\digit{j}{n}b^{-j}.
\end{equation*}
The \textit{Halton sequence} $(\bsy_n)_{n\geq 0}$
(in the bases $b_1,\ldots,b_s$) is given by 
\begin{equation*}
  \bsy_n=(  \phi_{b_1}(n),\ldots,  \phi_{b_s}(n))\in [0,1)^s,\quad n=0,1,2,\ldots .
\end{equation*}

It is well known (see~\cite[Theorem 3.6]{Niederreiter92}) that the discrepancy of the Halton sequence in pairwise coprime 
bases $b_1,\ldots,b_s \ge 2$ satisfies a low-discrepancy bound, i.e., 
\begin{equation} \label{Equ:Halton}
D_N(\bsy_n)=O_{b_1,\ldots,b_s}\({(\log\, N)^s}/{N}\) \qquad \mbox{for all } N \ge 2. 
\end{equation}

The one-dimensional Kronecker sequences $\left((\{n\alpha\})\right)_{n\geq 0}$ related to the real number $\alpha$ is uniformly distributed if $\alpha$ is irrational. This sequence is a low-discrepancy sequence if the continued fraction coefficients of $\alpha$ are bounded. Furthermore, the discrepancy $D_N$ satisfies $ND_N\leq C(\alpha,\epsilon) \log^{1+\epsilon} N$ for all $N$ and all $\epsilon>0$ for almost all $\alpha\in\RR$ in the sense of Lebesgue measure. Here the positive constant $C(\alpha,\epsilon)$ may depend on $\alpha$ and $\epsilon>0$ but is independent of $N$. 

Subsequences of the one-dimensional Kronecker sequence and Halton sequences are frequently studied objects (see for instance
\cite{AHL}, \cite{HellNied11}, \cite{HoRa}, \cite{HKLP}). Here we mention a result of Hofer and Ramare \cite{HoRa} who studied the sequence $(\lfloor n\alpha\rfloor \beta)_{n\geq 0}$. It is known to be u.d. mod
$1$ when $\alpha$ is a nonzero rational real if and only if $\beta$ is
irrational. When $\alpha=0$ the sequence is for no real $\beta$ u.d. mod
$1$. And when $\alpha$ is irrational the uniform distribution modulo one of
$([n\alpha]\beta)_{n\geq 0}$ is equivalent to the condition,
$1,\alpha,\alpha\beta$ being linearly independent over the rationals. (For a
proof see \cite[Chapter 5, Theorem 1.8]{Kuipers-Niederreiter*74}.) In \cite{HoRa} many results on the discrepancy of this sequence were proved. Here we mention just one result, i.e.,  for almost all pairs of real numbers $(\alpha,\beta)$ in the sense of Lebesgue
  measure we have for every $\epsilon>0$ that
  $D_N([n\alpha]\beta)\ll_{\alpha,\beta,\epsilon} N^{-1+\epsilon}$. 
%  Or, assume $\alpha$ and $\beta$ are two algebraic numbers such that $1$, 
%  $\alpha$, and $\alpha\beta$ are linearly independent over the rationals. Then for every $\epsilon>0$ we know that
% $$D_N([n\alpha]\beta)\ll_{\alpha,\beta,\epsilon}
%  N^{-1+\epsilon}.$$ 

In this paper we investigate the subsequence $(\bsx_n)_{n\geq 0}$ of the Halton sequence of this particular form
\begin{equation}\label{eq:defSequence}
\bsx_n:=(  \phi_{b_1}(\lfloor n\beta \rfloor),\ldots,  \phi_{b_s}(\lfloor n\beta \rfloor))\in [0,1)^s,\quad n=0,1,2,\ldots \end{equation}
with nonzero $\beta\in\RR$ and integers $b_1,\ldots,b_s\geq 2$ pairwise coprime. See Figure \ref{fig:1} for examples.  

\begin{figure}[t]\label{fig:1}
\caption{The first 500 points in bases $2$ and $3$ for $\beta=1,\sqrt{2},\pi$ and $e$.}
\includegraphics[width=0.5\textwidth]{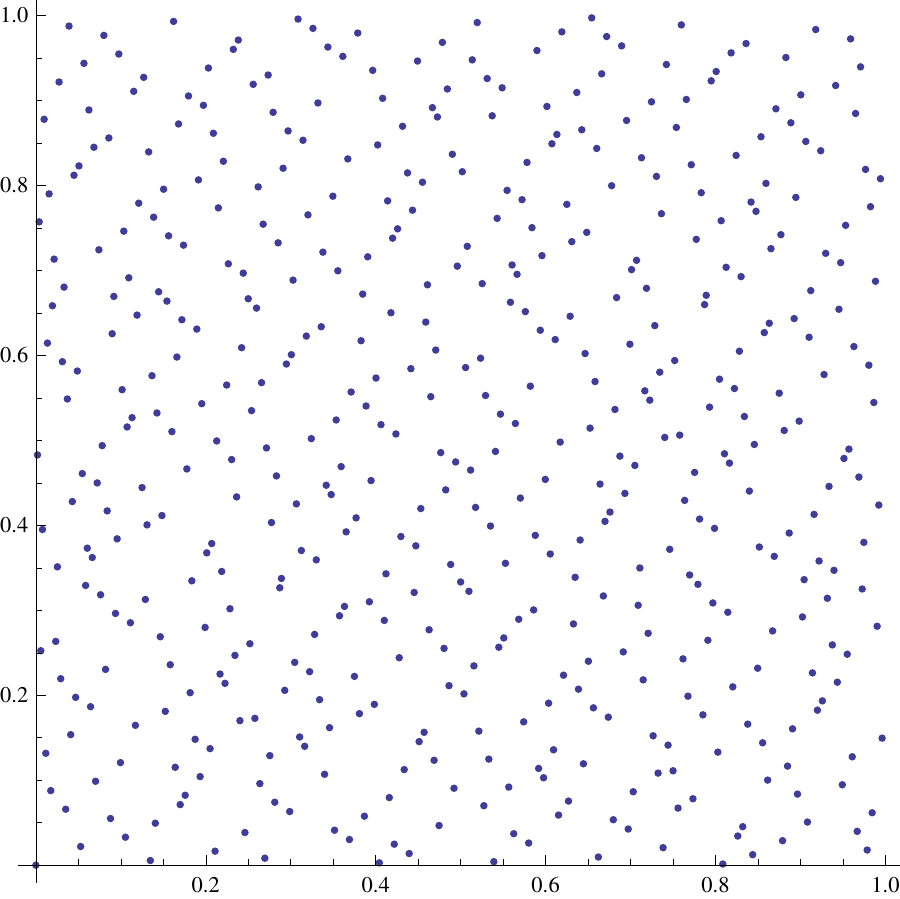}
\includegraphics[width=0.5\textwidth]{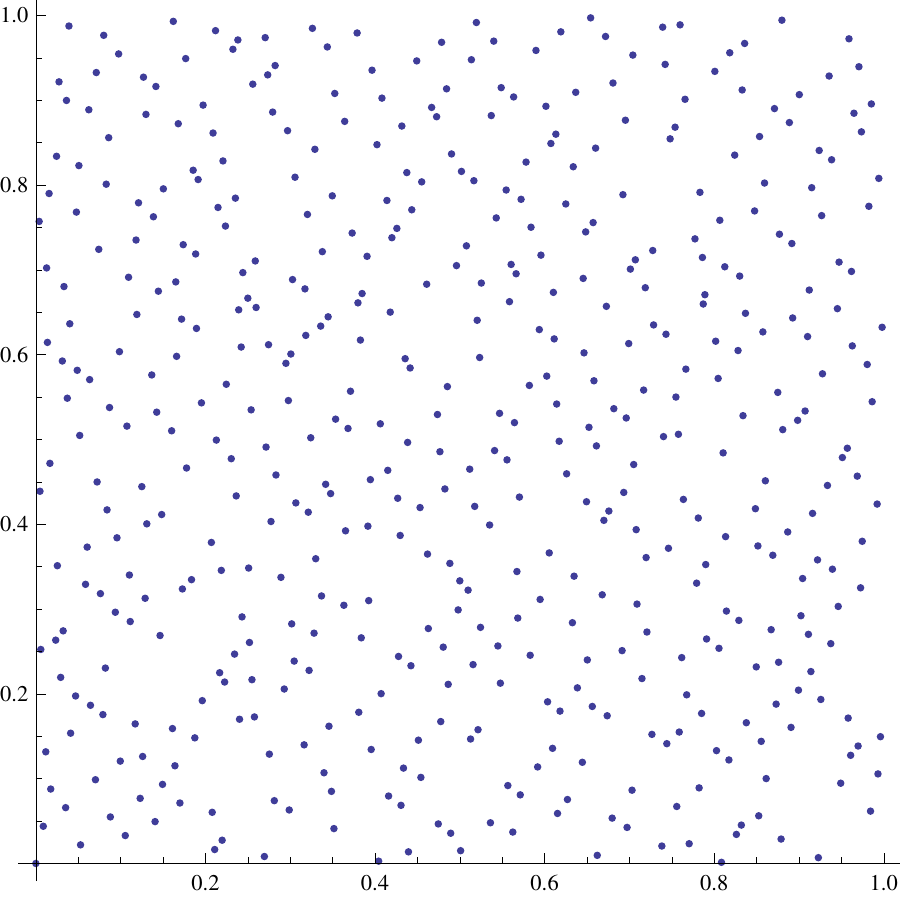}\\
\includegraphics[width=0.5\textwidth]{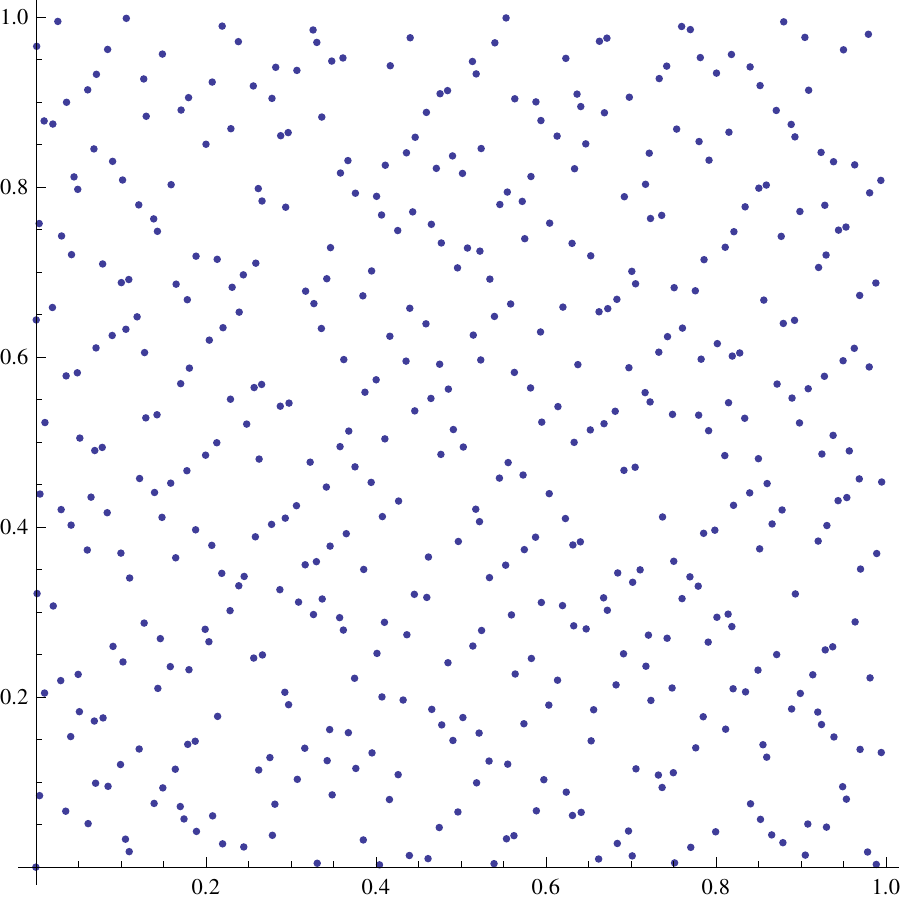}
\includegraphics[width=0.5\textwidth]{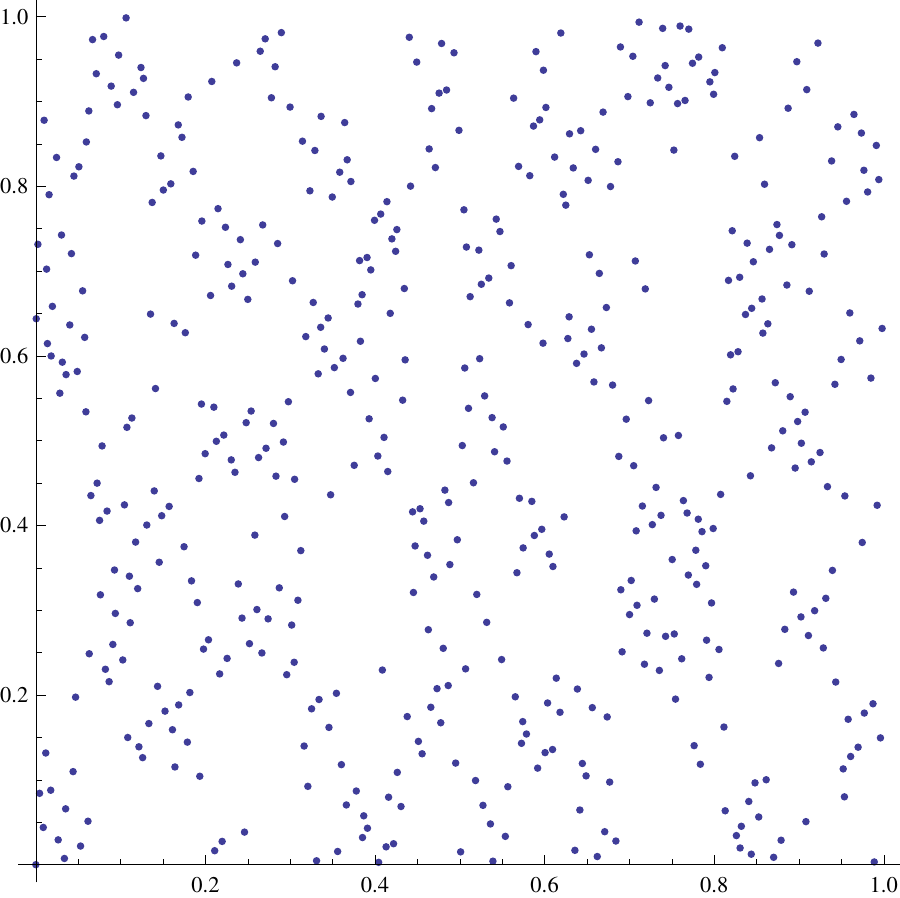}\\
\end{figure}

In \cite[Example~5.3]{HellNied11} the uniform distribution of \eqref{eq:defSequence} was studied. A sufficient criterion for the one-dimensional sequence is, for example, the linear independence of $1,\beta$ over $\QQ$. This criterion is not a necessary one, as the choice $\beta=1/d$ with $d\in\NN$ also yields a uniformly distributed sequence \eqref{eq:defSequence}. If $\beta=d\in\NN$ such that $(d,b_i)=1$ for $i=1,\ldots,s$ then the sequence \eqref{eq:defSequence} is even a low-discrepancy sequence (see \cite{HKLP}). \\

In this paper we prove that the sequence \eqref{eq:defSequence} is for almost all $\beta\in\RR$ an almost low-discrepancy sequence by the following theorem. 

\begin{theorem}\label{thm:1}
For almost all $\beta\in\RR$ in the sense of Lebesgue measure the discrepancy of $(\bsx_n)_{n\geq 0}$ satisfies for all $\epsilon>0$
$$ND_N= O_{\beta,\epsilon,b_1,\ldots,b_s,s}(\log^{s+1+\epsilon} {N}).$$
\end{theorem}

A proof of this theorem will be outsourced in the next section. 

%The restriction $\beta\in(0,1)$ is heavily used in this prove. Therefore, the metric discrepancy for $\beta>1$ remains an open problem. There are only a few results for $\beta>1$. 
%\begin{theorem}
%If $\alpha\in[0,1)$ is an algebraic irrational then the discrepancy of $(\bsx_n)_{n\geq 0}$ satisfies 
%$$ND_N= O_{\alpha,\epsilon}(N^{\epsilon}).$$
%\end{theorem}

\section{Proof of Theorem~\ref{thm:1}}

We start with the case $\beta\in(0,1)$. We write $\alpha:=1/\beta$. Note that $\alpha>1$. Then Theorem~\ref{thm:1} immediately follows from the following three auxiliary results. 

\begin{prop}  \label{prop:1} The star-discrepancy of the first $N$ points of the sequence \eqref{eq:defSequence} satisfies 

  \begin{align*}
ND^*_N(\bsx_n)\leq b_1\ldots b_s \sum_{j_1=0}^{f_1}\cdots \sum_{j_s=0}^{f_s}N\Delta^{(j_1,\ldots,j_s)}_N(\{n/(\alpha b_1^{j_1}\cdots b_s^{j_s})\})+s,
\end{align*}
where 
  \begin{equation} \label{Equ:deffi}
    f_i:=\left\lceil \frac{1}{\log b_i}\log\, N\right\rceil
    \quad\text{ for } 1\leq i\leq s
  \end{equation}
  and $$N\Delta^{(j_1,\ldots,j_s)}_N\left(\left\{\frac{n}{\alpha b_1^{j_1}\cdots b_s^{j_s}}\right\}\right):=\sup_{0\leq R< b_1^{j_1}\cdots b_s^{j_s}}\left|A_N\left(\left[\frac{R}{ b_1^{j_1}\cdots b_s^{j_s}},\frac{R+1}{ b_1^{j_1}\cdots b_s^{j_s}}\right)\right)-\frac{N}{ b_1^{j_1}\cdots b_s^{j_s}}\right|.$$
\end{prop}

\begin{proof}
  The first aim in the proof is to compute or
  estimate the function $A_N(J)-N\lambda_s(J)$ relative to the points $\bsx_n$ in~\eqref{eq:defSequence}, 
where $J\subseteq [0,1)^{s}$ is an interval of the form
\begin{equation} \label{Equ:formJ}
  J=\prod_{i=1}^s[0,v_ib_i^{-f_i})
\end{equation}
 with $v_1,\ldots,v_s \in\ZZ,\,
  1\leq v_i< b_i^{f_i}$ for $1\leq i\leq s$. We extend $v_i$ in base $b_i$, i.e., $$v_i=u^{(i)}_{f_i}+u^{(i)}_{f_i-1}b_i+\cdots+u^{(i)}_{1}b_i^{f_i-1}.$$
We can write $J$ as a union of $\prod_{i=1}^s(\sum_{j=1}^{f_i}u^{(i)}_{j})\leq (b_1-1) \cdots (b_s-1) f_1 \cdots f_s$ elementary intervals of the form 
$$\prod_{i=1}^s\left(\bigcup_{j=1}^{f_i}\bigcup_{k=1}^{u^{(i)}_{j}}\left[\sum_{l=1}^{j-1}\frac{u^{(i)}_l}{b_i^l} +\frac{k-1}{b_i^j}, \sum_{l=1}^{j-1}\frac{u^{(i)}_l}{b_i^l}+ \frac{k}{b_i^j}\right)\right)$$

  The crucial step is to estimate $A(I_e,N)-N\lambda_s(I_e)$ for an elementary interval $I_e$ by exploiting special properties of the Halton sequence. Using the definition of the sequence we obtain 
 for any integer $n\ge 0$ that the following three assertions are equivalent. 
 \begin{itemize}
 \item[-]  \begin{equation*}
    (\phi_{b_1}(\lfloor n/\alpha\rfloor),\ldots,\phi_{b_s}(\lfloor n/\alpha \rfloor))\in\prod_{i=1}^s\left[\sum_{l=1}^{j_i-1}\frac{u^{(i)}_l}{b_i^l} +\frac{k_i-1}{b_i^{j_i}}, \sum_{l=1}^{j_1-1}\frac{u^{(i)}_l}{b_i^l}+ \frac{k_i}{b_i^{j_i}}\right),
  \end{equation*}    
  where $1\leq j_i\leq f_i$ and $1\leq k_i\leq u^{(i)}_{j_i}$ for $i=1,\ldots,s$.
\item[-]
  \begin{equation} \label{equ:system}
  \lfloor n/\alpha\rfloor \equiv a_1 \pmod{b_1^{j_1}}, \ldots,  \lfloor n/\alpha\rfloor \equiv a_s \pmod{b_1^{j_s}}
    \end{equation}
    with $a_i=(k_i-1)b_i^{j_i-1}+\sum_{l=1}^{j_i-1}u_l^{(i)}b_i^{l-1}$. 
\item[-]  $$\lfloor n/\alpha\rfloor \equiv R \pmod{b_1^{j_1} \cdots b_s^{j_s}}$$
where the $R$ is uniquely determined by the $a_1,\ldots,a_s$. 
\end{itemize} 

By observing the fact that 
$$\lfloor n/\alpha\rfloor \equiv \lfloor n/\alpha-b_1^{j_1}\cdots b_s^{j_s}\lfloor n/(\alpha b_1^{j_1}\cdots b_s^{j_s})\rfloor \rfloor\equiv R \pmod{b_1^{j_1}\cdots b_s^{j_s}}$$
we easily obtain the equivalence of the above assertion to 
 \begin{equation*}
    \left\{\frac{n}{\alpha b_1^{j_1}\cdots b_s^{j_s}}\right\}\in\left[\frac{R}{b_1^{j_1}\cdots b_s^{j_s}},\frac{R+1}{b_1^{j_1}\cdots b_s^{j_s}}\right).
  \end{equation*}

Hence,

\begin{align*}
\left|A(J,N)-N\lambda_s(J)\right|\leq (b_1-1)\ldots (b_s-1) \sum_{j_1=1}^{f_1}\cdots \sum_{j_s=1}^{f_s}N\Delta^{(j_1,\ldots,j_s)}_N(\{n/(\alpha b_1^{j_1}\cdots b_s^{j_s})\})
\end{align*}

  An arbitrary interval $I\subseteq [0,1)^{s}$ of the form 
\begin{equation} \label{Equ:formI}
I= \prod_{i=1}^s[0,w_i)
\end{equation}
 with
  $0<w_i\leq 1$ for $1\leq i\leq s$ can be approximated from below
  by an interval $J$ of the form~\eqref{Equ:formJ},
  by taking the nearest fraction to the left of $w_i$ of the form 
$v_ib_i^{-f_i}$ with $v_i\in\ZZ$. We easily get
  \begin{equation*}
    \left|A(I,N) -N\lambda_{s}(I)\right|\leq 
    \left|A(J,N) -N\lambda_{s}(J)\right|+\max(A(I\setminus J,N), N\lambda_{s}(I\setminus J)).
  \end{equation*}
  The definition of the $f_i$ in~\eqref{Equ:deffi} yields $f_i\geq \log_{b_i}N$. Hence $N\lambda_{s}(I\setminus J)\leq N \sum_{i=1}^sb_i^{f_i}\leq s$. 
  It remains to estimate $A(I\setminus J,N)$: 
  We trivially have $A(I\setminus J,N)\leq A(S,N)\leq |A(S,N)-N\lambda_s(S)|+N\lambda_s(S)$ with $S$ a union of elementary intervals $I_e$ of the form 
\begin{align*}  S=&\bigcup_{k=1}^s\prod_{{i=1}}^{k-1}\left[0,1\right)   \times \left[\sum_{l=1}^{f_k-1}\frac{u^{(i)}_l}{b_i^l} +\frac{u_{f_k}^{(i)}}{b_i^{f_i}}, \sum_{l=1}^{f_k-1}\frac{u^{(i)}_l}{b_i^l}+ \frac{u_{f_k}^{(i)}+1}{b_i^{f_i}}\right) \times \prod_{{i=k+1}}^{s}\left[0,1\right).
\end{align*}
 Now $N\lambda_s(I_e)\leq 1$ by the definition of $f_k$ and $$|A(I_e,N)-N\lambda_s(I_e)|\leq N\Delta^{(0,\ldots,0,f_k,0,\ldots,0)}_N(\{n/(\alpha  b_k^{f_k})\})$$ and the result follows.  
  
\end{proof}

The discrepancy of a one-dimensional Kronecker sequence $(\{n\beta\})_{n\geq 0}$ can be related to the continued fraction expansion of $\beta$ 

$$\beta=[\lfloor \beta\rfloor;a_1(\beta),a_2(\beta),a_3\beta,\ldots ]=\lfloor \beta\rfloor+ \frac{1}{a_1(\beta)+\frac{1}{a_2(\beta)+\frac{1}{a_3(\beta)+\frac{1}{\ddots}}}}$$

This expansion is obtained by setting $\{\beta\}=x_1$ and defining inductively 
$$a_k(\beta)=\lfloor 1/x_k\rfloor \mbox{ and }x_{k+1}=\{1/x_k\}.$$ 

The convergents $p_k/q_k$ to $\beta$ are defined by $p_k/q_k=[\lfloor \beta\rfloor;a_1(\beta),a_2(\beta),\ldots,a_k(\beta) ]$. 
The denominators $q_k$ obey the following recursion 
$$q_{k}=a_k(\beta)q_{k-1}+q_{k-2}\mbox{ for $k\geq 1$ and with $q_0=1$, $q_{-1}=0$.}$$

\begin{lemma}\label{lem:1}
We have 
$$N\Delta^{(j_1,\ldots,j_s)}_N(\{n/(\alpha b_1^{j_1}\cdots b_s^{j_s})\})\leq \alpha+1+ 2\left( \sum_{k=1}^{\lceil\log_{3/2}(N)\rceil}a_k((\alpha b_1^{j_1}\cdots b_s^{j_s}))\right). $$
\end{lemma}
\begin{proof}

We follow the idea of \cite[page 125, proof of Theorem 3.4]{Kuipers-Niederreiter*74}. %We distinguish between two cases: $\alpha b_1^{j_1}\cdots b_s^{j_s}\geq 1$ and $\alpha b_1^{j_1}\cdots b_s^{j_s}<1$. 
Now we use the fact that $\beta\in(0,1)$ and hence $\alpha>1$. 
Then $\alpha b_1^{j_1}\cdots b_s^{j_s}>1$, 
and the first coefficient $a_1(1/(\alpha b_1^{j_1}\cdots b_s^{j_s}))=\lfloor \alpha b_1^{j_1}\cdots b_s^{j_s}\rfloor$. 
We compute the Ostrowski expansion using the denominators $1=q_0=q_1<q_2<q_3<\cdots$ of the convergents to $1/(\alpha b_1^{j_1}\cdots b_s^{j_s})$: $$N=N_0+N_1q_1+N_2q_2+\cdots+N_rq_r$$
with $q_r\leq N<q_{r+1}$ and $N_i\leq a_{i+1}(1/\alpha b_1^{j_1}\cdots b_s^{j_s})$.

We split 
$$N\Delta^{(j_1,\ldots,j_s)}_N(\{\frac{n}{\alpha b_1^{j_1}\cdots b_s^{j_s}}\})\leq N_0\Delta^{(j_1,\ldots,j_s)}_{N_0}(\{\frac{ n}{\alpha b_1^{j_1}\cdots b_s^{j_s}}\})+(N-N_0)\Delta^{(j_1,\ldots,j_s)}_{N-N_0}(\{\frac{n+N_0}{\alpha b_1^{j_1}\cdots b_s^{j_s}}\})$$
and proceed for the second term as in \cite[page 125, proof of Theorem 3.4]{Kuipers-Niederreiter*74} and obtain 
\begin{align*}
(N-N_0)\Delta^{(j_1,\ldots,j_s)}_{N-N_0}(\{\frac{n+N_0}{\alpha b_1^{j_1}\cdots b_s^{j_s}}\})&\leq 2\sum_{k=2}^{\lceil\log_{3/2} N\rceil+1} a_k(1/(\alpha b_1^{j_1}\cdots b_s^{j_s}))\\
&\leq 2\sum_{k=1}^{\lceil\log_{3/2} N\rceil} a_k(\alpha b_1^{j_1}\cdots b_s^{j_s}).
\end{align*}
Finally, it is not so hard to see that $N_0\Delta^{(j_1,\ldots,j_s)}_{N_0}(\{n/(\alpha b_1^{j_1}\cdots b_s^{j_s})\})\leq \alpha+1$ with $N_0\leq \lfloor \alpha b_1^{j_1}\cdots b_s^{j_s}\rfloor$. Note the fact that $[R/b_1^{j_1}\cdots b_s^{j_s}),(R+1)/b_1^{j_1}\cdots b_s^{j_s})$ either remains empty or contains at most $\lfloor \alpha \rfloor+1$ points, and $N_01/b_1^{j_1}\cdots b_s^{j_s}\leq\lfloor \alpha \rfloor+1$.

\end{proof}

%For the proof of Theorem~\ref{thm:1} we make use of the following auxiliary result. 

\begin{lemma}\label{lem:2}
Let $\alpha\in\RR$. For any $L\in\NN$ let 
$$S_L(\alpha):=\sum_{j_1=0}^L\cdots \sum_{j_s=0}^L\sum_{k=1}^La_k(\alpha b_1^{j_1}\cdots b_s^{j_s}).$$
Then for almost all $\alpha\in\RR$ we have for every $\epsilon>0$ 
$$S_L(\alpha)=O_{\alpha,\epsilon,b_1,\ldots,b_s,s}(L^{s+1+\epsilon}).$$
\end{lemma}
\begin{proof}
See e.g. \cite[Lemma 2]{LarJOC}. 
\end{proof}

Finally we consider $\beta>1$ and set $\alpha:=\beta/(\beta+1)$. Then $1/2<\alpha<1$, $\lfloor n\alpha\rfloor=n-\lceil n\frac{1}{\beta +1}\rceil$ and hence 
$$\{\lfloor n\alpha\rfloor;n\in\NN\}=\{\lfloor n\beta\rfloor;n\in\NN\}\cup\NN_0.$$ 

Let $$B:=\left\{\alpha\in(1/2,1):D_N(\phi_{b_1}(\lfloor n\alpha\rfloor),\ldots,\phi_{b_s}(\lfloor n\alpha\rfloor))=O_{\alpha,\epsilon,b_1,\ldots,b_s}\left(\frac{(\log N)^{s+1+\epsilon}}{N}\right)\mbox{ for all }\epsilon>0\right\}.$$
From the above case we know $\lambda(B)=1/2$. Let $A:=\{\beta\in(1,\infty):\beta/(\beta+1)\in B\}$. Then $\lambda((1,\infty)\setminus A)=0$. 

Now let $\beta\in A$. Then $\alpha\in B$. 
Since 
$$\underbrace{(\phi_{b_1}(\lfloor n\alpha\rfloor),\ldots,\phi_{b_s}(\lfloor n\alpha\rfloor))_{n\geq 1}}_{D_N=O((\log N)^{s+1+\epsilon}/N)}=(\phi_{b_1}(\lfloor n\beta\rfloor),\ldots,\phi_{b_s}(\lfloor n\beta\rfloor))_{n\geq 1}\,\,\cup\,\, \underbrace{(\phi_{b_1}(n),\ldots,\phi_{b_s}(n))_{n\geq 0}}_{D_N=O((\log N)^s/N)}$$
we have 
$$D_N(\phi_{b_1}(\lfloor n\beta\rfloor),\ldots,\phi_{b_s}(\lfloor n\beta\rfloor))=O_{\beta,\epsilon,b_1,\ldots,b_s}\left(\frac{(\log N)^{s+1+\epsilon}}{N}\right)\mbox{ for all }\epsilon>0$$
and the proof is complete.

\section*{Acknowledgments}

The author would like to thank Gerhard Larcher for valuable comments and fruitful discussions.

Roswitha Hofer\\
Institute of Financial Mathematics and Applied Number Theory\\
Johannes Kepler University Linz\\
Altenbergerstr. 69\\
A-4040 Linz, AUSTRIA\\ 
roswitha.hofer@jku.at\\


\begin{thebibliography}{99}

\bibitem{AHL} 
C.~Aistleitner, R.~Hofer, and G.~Larcher, \textit{On evil Kronecker sequences and Lacunary Trigonometric Products}, to appear in Annales de l\'institut Fourier. (arXiv:1502.06738)


%\bibitem{hell} {\sc HELLEKALEK, P.} \emph{A general discrepancy bound based on $p$-adic arithmetics.} Acta Arith. \textbf{139}(2) (2009), 117--129. 

\bibitem{HellNied11}
{ P. Hellekalek and H. Niederreiter}
\emph{Constructions of uniformly distributed sequences using the {\bf b}-adic
  method}, Unif. Distrib. Theory {\bf 6} (2011), 185--200.

\bibitem{HoRa}
{ R. Hofer and O. Ramare,} \emph{Discrepancy estimates for some linear generalized monomials,} \emph{Acta Arith. 173}, 183--196, 2016.

\bibitem{HKLP} R.~Hofer, P.~Kritzer, G.~Larcher, F.~Pillichshammer.
\newblock {\em Distribution properties of generalized van der Corput-Halton
  sequences and their subsequences,}
\newblock {\em Int. J. Number Theory} 5(4), 719--746, 2009.
 
\bibitem{Kuipers-Niederreiter*74}
 L.~Kuipers and H.~Niederreiter, 
{\em Uniform distribution of sequences},
Wiley, New York, 1974; reprint, Dover Publ., Mineola, NY, 2006. 

\bibitem{LarJOC} G. Larcher, \emph{Probabilistic diophantine approximation and the distribution of Halton--Kronecker sequences}, \emph{J. Complexity 29}, 397--423. 

\bibitem{Niederreiter92}
H. Niederreiter, \emph{Random Number Generation and Quasi-Monte Carlo Methods},
SIAM, Philadelphia, 1992.


\end{thebibliography}
\end{document}